\documentclass[a4paper,12pt,reqno]{amsart}
\usepackage{amsfonts, color}
\usepackage{amsmath}
\usepackage{amssymb}
\usepackage[a4paper]{geometry}
\usepackage{mathrsfs}
\usepackage[colorlinks]{hyperref}
\usepackage{comment}
 \usepackage{verbatim}
\usepackage[colorinlistoftodos]{todonotes}

\usepackage{hyperref}
\renewcommand\eqref[1]{(\ref{#1})}

\makeatletter
\newcommand*{\mint}[1]{%
  \mint@l{#1}{}%
}
\newcommand*{\mint@l}[2]{%
  \@ifnextchar\limits{%
    \mint@l{#1}%
  }{%
    \@ifnextchar\nolimits{%
      \mint@l{#1}%
    }{%
      \@ifnextchar\displaylimits{%
        \mint@l{#1}%
      }{%
        \mint@s{#2}{#1}%
      }%
    }%
  }%
}
\newcommand*{\mint@s}[2]{%
  \@ifnextchar_{%
    \mint@sub{#1}{#2}%
  }{%
    \@ifnextchar^{%
      \mint@sup{#1}{#2}%
    }{%
      \mint@{#1}{#2}{}{}%
    }%
  }%
}
\def\mint@sub#1#2_#3{%
  \@ifnextchar^{%
    \mint@sub@sup{#1}{#2}{#3}%
  }{%
    \mint@{#1}{#2}{#3}{}%
  }%
}
\def\mint@sup#1#2^#3{%
  \@ifnextchar_{%
    \mint@sup@sub{#1}{#2}{#3}%
  }{%
    \mint@{#1}{#2}{}{#3}%
  }%
}
\def\mint@sub@sup#1#2#3^#4{%
  \mint@{#1}{#2}{#3}{#4}%
}
\def\mint@sup@sub#1#2#3_#4{%
  \mint@{#1}{#2}{#4}{#3}%
}
\newcommand*{\mint@}[4]{%
  \mathop{}%
  \mkern-\thinmuskip
  \mathchoice{%
    \mint@@{#1}{#2}{#3}{#4}%
        \displaystyle\textstyle\scriptstyle
  }{%
    \mint@@{#1}{#2}{#3}{#4}%
        \textstyle\scriptstyle\scriptstyle
  }{%
    \mint@@{#1}{#2}{#3}{#4}%
        \scriptstyle\scriptscriptstyle\scriptscriptstyle
  }{%
    \mint@@{#1}{#2}{#3}{#4}%
        \scriptscriptstyle\scriptscriptstyle\scriptscriptstyle
  }%
  \mkern-\thinmuskip
  \int#1%
  \ifx\\#3\\\else_{#3}\fi
  \ifx\\#4\\\else^{#4}\fi
}
\newcommand*{\mint@@}[7]{%
  \begingroup
    \sbox0{$#5\int\m@th$}%
    \sbox2{$#5\int_{}\m@th$}%
    \dimen2=\wd0 %
    \let\mint@limits=#1\relax
    \ifx\mint@limits\relax
      \sbox4{$#5\int_{\kern1sp}^{\kern1sp}\m@th$}%
      \ifdim\wd4>\wd2 %
        \let\mint@limits=\nolimits
      \else
        \let\mint@limits=\limits
      \fi
    \fi
    \ifx\mint@limits\displaylimits
      \ifx#5\displaystyle
        \let\mint@limits=\limits
      \fi
    \fi
    \ifx\mint@limits\limits
      \sbox0{$#7#3\m@th$}%
      \sbox2{$#7#4\m@th$}%
      \ifdim\wd0>\dimen2 %
        \dimen2=\wd0 %
      \fi
      \ifdim\wd2>\dimen2 %
        \dimen2=\wd2 %
      \fi
    \fi
    \rlap{%
      $#5%
        \vcenter{%
          \hbox to\dimen2{%
            \hss
            $#6{#2}\m@th$%
            \hss
          }%
        }%
      $%
    }%
  \endgroup
}
\numberwithin{equation}{section}
\theoremstyle{plain}
\newtheorem{thm}{Theorem}[section]
\newtheorem{prop}[thm]{Proposition}
\newtheorem{cor}[thm]{Corollary}

\theoremstyle{definition}

\newtheorem{rem}[thm]{Remark}

\newcommand{\G}{\mathbb G}
\newcommand{\R}{\mathbb{R}^n}
\usepackage{xcolor}

\long\def\symbolfootnote[#1]#2{\begingroup
\def\thefootnote{\fnsymbol{footnote}}\footnote[#1]{#2}\endgroup}

\numberwithin{equation}{section}

\numberwithin{equation}{section}

\begin{document}
\title[Anisotropic R\'enyi entropy and Heisenberg uncertainty principle] {Sharp upper bound for anisotropic R\'enyi entropy and Heisenberg uncertainty principle}
\author{Marianna Chatzakou, Michael Ruzhansky and Anjali Shriwastawa}

\address[Marianna Chatzakou]{\endgraf Department of Mathematics: Analysis,  Logic and Discrete       Mathematics, \endgraf Ghent University, Ghent, Belgium
}
\email{marianna.chatzakou@ugent.be}

\address[Michael Ruzhansky]{\endgraf Department of Mathematics: Analysis,  Logic and Discrete       Mathematics, \endgraf Ghent University, Ghent, Belgium
\endgraf and\endgraf
School of Mathematical Sciences, Queen Marry University of London, \endgraf United Kingdom}
\email{michael.ruzhansky@ugent.be}

\address[
 Anjali Shriwastawa]{
 \endgraf
  DST-Centre for Interdisciplinary Mathematical Sciences  \endgraf
  Banaras Hindu University, Varanasi-221005, India}
\email{anjalisrivastava7077@gmail.com}

\thanks{{\em 2020 Mathematics Subject Classification: } 26D10, 22E30, 45J05.}

\keywords{R\'enyi entropy; Shannon entropy; Shannon inequality;  Folland-Stein-homogeneous groups; Heisenberg uncertainty principle}

\maketitle

\begin{abstract} 
 In this paper, we prove the  anisotropic Shannon inequality for the R\'enyi entropy with the best constant on Folland-Stein homogeneous Lie groups.  As a consequence, we also prove the optimal Shannon inequality in the same setting.  Using a logarithmic Sobolev inequality in the setting of stratified groups, we prove a Heisenberg-type uncertainty principle in the latter setting.
\end{abstract}

%\tableofcontents

\section{Introduction}
Shannon and R\'enyi entropies were initially served as characteristics of probability distributions. To begin with, the Shannon entropy was introduced by Shannon in \cite{Shannon}, and  for a discrete random variable $X$ with distribution $p(x)$ it is given via 
\[
h[X]=-\sum_i p(x_i)\log p(x_i)\,.
\]
 In the continuous setting for a continuous random variable $X$ with (probability) density $u(x)$, the natural analog of the previous definition is as follows:
\begin{align}\label{shannon}
    h[X]:=h[u]=-\int_{\mathbb{R}^n}u(x)\log u(x)\,dx\,.
\end{align}
Clearly the Shannon entropy \eqref{shannon} makes sense for a nonnegative $u \in L^1(\mathbb{R}^n)$ such that $\|u\|_{L^1(\mathbb{R}^n)}=1$. In the sequel we will write $h[u]$ instead of $h(X)$, to denote the Shannon entropy of the continuous random variable $X$ with density $u$. Let us point out that in the continuous case that we study here, $p(x)$ is not a probability, but a probability density, and the latter are not the same. Indeed, the quantity $h[X]$, for $X$ being a discrete random variable, is always non-positive, while if $X$ is the continuous random variable which is uniformly distributed over the interval $(a,b)$; that is 
\[
u(x)=\begin{cases}
    \frac{1}{b-a}\,,\quad \text{if}\quad x \in (a,b)\\
    0\,,\quad \quad \text{otherwise}\,,
\end{cases}
\]
then $h[X]=h[u]=-\log(b-a)$, and clearly $h[u]>0$ if $b-a<1$, see e.g. \cite{Rioul}. Therefore, in the continuous case, one cannot, strictly speaking, talk about the ``amount of information'' represented by the entropy.

There are other features of the Shannon continuous  entropy that make its study a quite involved subject; see for example \cite{M13}. Despite these difficulties, continuous informatics is a field of wide interest, c.f. the monograph of Ihara \cite{Ihara}, or the work of Conrad \cite{Conrad}. Particularly, in relation to partial differential equations, the conclusion of the Boltzman H-Theorem, see e.g. \cite{Villani}, holds true for  the solution to the heat equation on $\mathbb{R}^n$
\begin{equation}
    \label{heat.eq}
    \partial_t u=\Delta u\,, \quad t>0\,;
\end{equation}
that is the Shannon entropy \eqref{shannon} is not decreasing in time, since 
\[
\frac{d}{dt}h[u(x,t)]=\int_{\mathbb{R}^n}\frac{1}{u(x,t)}|\nabla u(x,t)|^2\, dx \geq 0\,.
\]
Analogously, the R\'enyi entropy, introduced by R\'enyi in \cite{Renyi}, is for nonnegative function $u \geq 0$ with $\|u\|_{L^1(\mathbb{R}^n)}=1$, given by 
\begin{equation}
\label{renyi}
h_\alpha[u]=\frac{1}{1-\alpha} \log \int_{\mathbb{R}^n}u(x)^\alpha\,dx\,,
\end{equation}
``corresponds'' to the nonlinear diffusion equations 
\[
\partial_t u= \Delta u^\alpha\,,\quad t>0, \alpha>0, \alpha \neq 1\,.
\]
It is easy to check that the R\'enyi entropy \eqref{renyi} is an extension of the Shannon entropy \eqref{shannon} for different values of $\alpha$ in the sense that 
\[
\lim_{\alpha\rightarrow 1}h_{\alpha}[u]=h[u]\,,
\]
for any function $u$ as above. Both Shannon and R\'enyi entropies, in both discrete and continuous settings, are monotone in the sense that they increase whenever an independent random variable is added, and hence a natural question is whether one can find an upper bound for them; in the Euclidean setting $\mathbb{R}^n$, this is given by Shannon's inequality \cite{Shannon, Suguro} which is an upper bound  of the form:
\begin{equation}
    \label{shannon.in.intro}
    h[u] \leq \frac{n}{2} \log \left(\frac{2 \pi e}{n} \int_{\mathbb{R}^n}|x|^2 u(x)\,dx \right)\,.
\end{equation}
Hence inequality \eqref{shannon.in.intro}  shows that the Shannon entropy of a function $u$  is bounded by the second moment of $u$, and the constant $\frac{2\pi e}{n}$ is the best possible. In the case of the R\'enyi entropy, this upper bound reads as follows:
\begin{equation}
    \label{renyi.in.intro}
    h_{\alpha}[u]\leq \frac{n}{b} \log \left(C_b \int_{\mathbb{R}^n}|x|^b u(x)\,dx \right)\,,\quad \alpha \in \mathbb{R}^{+}\setminus \{1\}\,;
\end{equation}
that is the R\'enyi entropy of the function $u$ is bounded by the $b^{\text{th}}$-moment of $u$, where $b>0$ depends on the range of $\alpha$. Both inequalities \eqref{shannon.in.intro} and \eqref{renyi.in.intro}  are valid for functions $u$ in some suitable weighted Lebesgue spaces. Inequality \eqref{renyi.in.intro} is sharp, and the constant $C_b$ in \eqref{renyi.in.intro} depends on the parameter $b$ on the right-hand side of \eqref{renyi.in.intro}, which in turn depends on the parameter $\alpha$ and is given explicitly in \cite{Suguro} where the aforementioned inequality is proved. For the so-called Shannon inequality in the Euclidean setting \eqref{shannon.in.intro} we refer to the paper \cite{Ogawa}, while Shannon's inequality has also been proved in the general  setting of homogeneous Lie groups in \cite{CKR}. Properties of both entropies have been widely studied; see e.g. \cite{BMM} and references therein. 

Applications of Shannon's inequality \eqref{shannon.in.intro} and Shannon's inequality for the R\'enyi entropy also include the positivity of the relative Lyapunov functional
\[
H_\alpha[u|v]=H_\alpha[u]-H_\alpha[v]\,,
\]
where $H_\alpha[\cdot]$ is the Lyapunov functional, see e.g. \cite{New}, given by 
\[
H_\alpha[u]=\frac{1}{\alpha-1}\int_{\mathbb{R}^{n}}u(x)^\alpha \,dx +\frac{1}{2}\int_{\mathbb{R}^{n}}|x|^2 u(x)\,dx\,,
\]
see \cite{Suguro} and \cite{Toscani}. Importantly, the combination of Shannon inequality \eqref{shannon.in.intro} together with the version of logarithmic Sobolev inequality thanks to Stam \cite{Stam} 
\begin{equation}
    \label{Stam}
    h[u] \geq -\frac{n}{2} \log \left(\frac{1}{2 n \pi e}\int_{\mathbb{R}^n} \frac{1}{u(x)}|\nabla u(x)|^2\,dx \right)\,,
\end{equation}
which holds true for non-negative $u \in L^1(\mathbb{R}^n)$ with $u^{1/2} \in H^1(\mathbb{R}^n)$, gives rise to the Heisenberg uncertainty principle; that is for such $u$ we have 
\begin{equation}
    \label{heisenberg.prin}
    n\leq \left(\int_{\mathbb{R}^n}|x|^2 u(x)\,dx \right)^{\frac{1}{2}}\left(\int_{\mathbb{R}^n}\frac{1}{u(x)}|\nabla u(x)|^2\,dx \right)^{\frac{1}{2}}\,,
\end{equation}
where $n$ is the best possible. Still in the Euclidean setting, Carillo and Toscani \cite{CT} proved the logarithmic Sobolev-type inequalities for the R\'enyi entropy; see also \cite{ST}.

The setting of the current work is that of homogeneous Lie groups introduced in Section \ref{preli}. In this setting, in \cite{CKR} the authors proved the (anisotropic) Shannon inequality, with an explicit constant that is sharp, see the discussion that follows after Remark \ref{rem.notsharp}. Here we prove the  analogue of Shannon's inequality for the R\'enyi entropy as in \eqref{renyi.in.intro} for the setting of homogeneous Lie groups, see Theorem \ref{mainresult}. For the proof of the latter, we follow  the lines in \cite{Suguro} and we prove the aforementioned inequality with the best constant as it happens in \cite{Suguro}, see Corollary \ref{cor.sharp}. Later on, in Theorem \ref{thm.Shann}, we prove the Shannon inequality as in \eqref{renyi.in.intro} in the setting of homogeneous Lie groups, using the fact that the Shannon entropy is the limiting case when $\alpha \rightarrow 1$ of the R\'enyi entropy. The Shannon inequality in the aforesaid setting was also proved in \cite{CKR} using different methods and the constant there coincides with ours since both are optimal. Our final result is the Heisenberg uncertainty principle in the setting of stratified Lie groups, see Corollary \ref{cor.unc}. The latter is proved by combining Shannon's inequality with a version of logarithmic Sobolev inequality in the aforesaid setting  that is  derived by the logarithmic Sobolev inequality as in \cite{CKR1}; see Theorem \ref{thm.rev.LS}.

\section{Preliminaries} \label{preli}
In this section, we give a brief description of our setting of homogeneous Lie groups. Such groups were initiated by Folland and Stein, see \cite{FS}, and later on, they became a subject of study by many authors; see \cite{FR} and references therein.  The notation that we adopt in the sequel follows the more recent open access monograph on homogeneous Lie groups \cite{FR}.

A connected, simply connected Lie group $\mathbb{G}\cong \mathbb{R}^N$ whose Lie algebra $\mathfrak{g}$ admits a gradation of the form \begin{equation}\label{strat}\mathfrak{g}=\oplus_{j=1}^{\infty}V_{j}\,,\end{equation} where finitely many $V_j$'s are nonzero and satisfy relations of the form $[V_{i},V_{j}]\subset V_{i+j}$ is called a graded Lie group. Graded Lie groups are naturally homogeneous Lie groups meaning that there exists a dilation mapping denoted as $D_\lambda: \mathbb{R}^N \rightarrow \mathbb{R}^N$, where $\lambda > 0$, that is an automorphism of $\mathbb{G}$, and so also of $\mathfrak{g}$. Particularly, for $\lambda>0$, the mapping $D_\lambda$ acts on $x \in \mathbb{G}$ via
\begin{equation}
D_\lambda(x)= (\lambda^{v_1}x_1,\lambda^{v_2}x_2,\ldots,\lambda^{v_N}x_N) ,\quad v_1,v_2,\dots,v_N > 0\,.
\end{equation}

The so-called weights $v_1,\cdots, v_n$ determine the homogeneous dimension, usually denoted by $Q$,  of the homogeneous group $\mathbb{G}$ in the following way:
$$Q=v_1+v_2+...+v_N.$$

A homogeneous group is unimodular and the unique Haar measure denoted by $dx$ on $\mathbb{G}$ is the Lebesgue measure on the underlying manifold $\mathbb{R}^N$. Hence for $\omega \subset \mathbb{G}$ being a measurable set in $\mathbb{G}$, if  $|\omega|$ stands for the volume of $\omega$, then for $\lambda > 0$ we have
\begin{equation}
|D_\lambda(\omega)|=\lambda^Q |\omega| \quad \text{and} \quad \int_{\mathbb{G}} f(D_\lambda (x)) dx = \lambda^{-Q}\int_{\mathbb{G}} f(x) dx.
\end{equation}

For any homogeneous Lie group $\mathbb{G}$, there exists a homogeneous (with respect to the dilations determined by $D_\lambda$ on $\G$) quasi-norm $|\cdot|$ defined on $\mathbb{G}$; that is $|\cdot|:\mathbb{G} \rightarrow [0,\infty)$ is  a continuous, non-negative function  that satisfies the following conditions:
\begin{itemize}
\item[(i)] For all $x \in \mathbb{G}$, $|x|=|x^{-1}|$.
\item[(ii)] For all $x \in \mathbb{G}$ and $\lambda > 0$, $|\lambda x|=\lambda |x|$.
\item[(iii)] $|x|=0$ if and only if $x=0$.
\end{itemize}
We note that in the sequel we use the notation $|\cdot|$ to denote both the volume of a measurable set in $\mathbb{G}$ and the homogeneous quasi-norm of an element $x \in \mathbb{G}$, and the meaning of $|\cdot|$ is each appearance will be clear from the context.

%\textbf{Polar decomposition in homogeneous Lie groups:}
%Consider a homogeneous Lie group $\mathbb{G}$ equipped with the homogeneous quasi-norm $|\cdot|$. Let $\mathfrak{S} := \{x \in \mathbb{G}: |x| = 1\}$ represent the unit sphere in $\mathbb{G}$. It follows that there exists a unique Radon measure $\sigma$ defined on $\mathfrak{S}$ such that, for any $u \in L^1(\mathbb{G})$, the following holds:
%\begin{align}
 %   \int_\mathbb{G}u(x)\,dx=\int_0^\infty\int_\mathfrak{S}u(ry)\,r^{Q-1}\,d\sigma(y)\,dr.
%\end{align}
Finally, let us introduce the following Lebesgue spaces that are useful for our purposes: For $a>0$, the weighted Lebesgue space denoted by $L^{1}_{a}(\mathbb{G})$ is defined as follows:
\[
L_{a}^{1}(\mathbb{G})=\{u \in L^{1}(\mathbb{G}) \quad \text{such that}\quad |x|^{a}u\in L^1(\mathbb{G}) \}\,,
\]
where $|x|$ stands for the quasi-norm of $x \in \mathbb{G}$.
% The weighted Lebesgue space on homogeneous Lie group is, denoted as $L^{p,\alpha}(\mathbb{G})$, where $\alpha>0$. It can be expressed as follows:

%\begin{align*}
%L^{p,\alpha}(\mathbb{G}) := \left\{ u \in L^p_{\text{loc}}(\mathbb{G})  ;  \langle x\rangle^\alpha u\in L^p(\mathbb{G}) \right\}.
%\end{align*}
%In the above expression $\langle x \rangle=\sqrt{1+|x|^2}$ for all $x\in \mathbb{G}$, and $|\cdot|$ represents a quasi-norm on a homogeneous Lie group $\mathbb{G}$.\\
%Now, we will recall Jensen's inequality in the context of homogeneous Lie groups, as it will be employed in the proof of the main result of this paper.\\
%\textbf{Jensen's inequality for convex functions on $\mathbb{G}$}:
%Let $\mathbb{G}$ be a homogeneous group equipped with a measure $\mu$, and let $f: \mathbb{G} \rightarrow \mathbb{R}$ be a measurable function. If $\mu$ is a probability measure on $\mathbb{G}$, and $\Phi:\mathbb{R}\rightarrow\mathbb{R}$ is a convex function, then we have:
%\begin{align}
%\Phi\left(\int_\Omega f d\mu\right)\le \int_\Omega \Phi o f d\mu.
%\end{align}
%\textbf{Jensen's inequality for concave functions on $\mathbb{G}:$}
%Let $\mathbb{G}$ be a homogeneous group equipped with a measure $\mu$, and let $f: \mathbb{G} \rightarrow \mathbb{R}$ be a measurable function. If $\mu$ is a probability measure on $\mathbb{G}$, and $\Phi:\mathbb{R}\rightarrow\mathbb{R}$ is a convex function, then  we have:
%\begin{align}
%\Phi\left(\int_\Omega f d\mu\right)\ge \int_\Omega \Phi o f d\mu.
%\end{align}
\section{Main Results}\label{mainresult}
In this section, we establish the main results of the paper starting with proving the Shannon inequality for the R\'enyi entropy in the setting of homogeneous Lie groups.
\begin{thm} \label{mainthm} Let $\mathbb{G}$ be a homogeneous Lie group of homogeneous dimension $Q$, and let $|\cdot|$ be a homogeneous quasi-norm on $\G$.  Suppose that $\alpha>0$, $\alpha \neq 1$, and 
$$b> \begin{cases} 
Q\left(\frac{1}{\alpha}-1 \right)\,,\quad &\text{if}\quad 0<\alpha<1,\\0\,, \quad &\text{if}\quad \alpha>1.
\end{cases}$$
Then, for any nonnegative function $u \in L^1_b(\G)$ with $\| u\|_{L^1(\G)}=1$,  the inequality
\begin{equation}\label{eq1.1}
    \frac{1}{1-\alpha} \log \int_\G u(x)^\alpha\, dx \leq \frac{Q}{b} \log \left( A_{\alpha, Q,b}^{\frac{b}{Q}} \int_\G |x|^b u(x) dx \right),
\end{equation}
holds, where 
\begin{align*} A_{\alpha, Q,b}:=\begin{cases}
  \left( \frac{b}{|\mathfrak{S}|}\frac{\Gamma\left(\frac{1}{1-\alpha}\right)}{\Gamma\left(\frac{1}{1-\alpha}-\frac{Q}{b}\right)\Gamma(\frac{Q}{b})} \right)^{-1}\left(\frac{\alpha b}{\alpha b-Q(1-\alpha)}\right)^{\frac{1}{1-\alpha}}\left(\frac{\alpha b-Q(1-\alpha)}{Q(1-\alpha)}\right)^\frac{Q}{b}\,, &\text{if}\quad 0<\alpha<1,\\ \frac{\alpha b}{Q(\alpha-1)}\left(\frac{\alpha b +Q(\alpha-1)}{\alpha b} \right)^{\frac{Q(\alpha-1)+b}{Q(\alpha-1)}}\left(\frac{b}{|\mathfrak{S}|}\frac{\Gamma\left(\frac{Q}{b} \right) \Gamma\left(\frac{\alpha}{\alpha-1} \right)}{\Gamma\left( \frac{\alpha}{\alpha-1}+\frac{Q}{b}\right)}\right)^{-\frac{b}{Q}}\,,\quad &\text{if}\quad\alpha>1\,,
\end{cases}
\end{align*} 
 and $|\mathfrak{S}|$ stands for the $Q-1$ dimensional surface measure of the unit (quasi-)sphere with respect to $|\cdot|.$
\end{thm}
\begin{proof}[Proof of Theorem \ref{mainthm}]
 Note that it is enough to prove \eqref{eq1.1} for a nonnegative function $u \geq 0$ that is smooth on $\G$. The result will then follow by a density argument. We will treat the cases where $\alpha \in (0,1)$ and $\alpha>1$ separately. Let us first consider the case $\alpha \in (0,1)$. For such values of $\alpha$ and for $b>Q\left( \frac{1}{\alpha}-1\right)$ we consider the auxiliary function $\phi_1$ given by 
$$\phi_1(x)=C_{1}\left(1+|x|^b\right)^{\frac{1}{\alpha-1}}\,,$$ with $C_1$ computed in the Appendix, see \eqref{C_1}.  Using Jensen's inequality for the convex function $\log 1/t$ and the the probability measure $\frac{u(x)^{\alpha}}{\|u\|^{\alpha}_{L^{\alpha}(\G)}}dx$ we estimate from above the relative entropy of $u$ and $\phi_1$ as follows
\begin{equation}\label{1}
\begin{split}
    \int_{\mathbb{G}} u(x)^\alpha \log \frac{\phi_1(x)}{u(x)}dx &\ = \frac{\|u\|^{\alpha}_{L^{\alpha}(\G)}}{\alpha}  \int_{\G} \frac{u(x)^{\alpha}}{\|u\|^{\alpha}_{L^{\alpha}(\G)}}\log \frac{\phi_1(x)^{\alpha}}{u(x)^{\alpha}}dx \\ &
     \leq  \frac{\|u\|^{\alpha}_{L^{\alpha}(\G)}}{\alpha} \log \left( \frac{1}{\|u\|^{\alpha}_{L^{\alpha}(\G)}} \int_{\G} \phi_1(x)^{\alpha}\,dx\right)\\ &
     =  \frac{\|u\|^{\alpha}_{L^{\alpha}(\G)}}{\alpha} \log \frac{\|\phi_1\|^{\alpha}_{L^{\alpha}(\G)}}{\|u\|^{\alpha}_{L^{\alpha}(\G)}}\,,
    \end{split}
\end{equation}
where $\|\phi_1\|^{\alpha}_{L^{\alpha}(\G)}$ is computed in Appendix, see \eqref{phia}, as 
\[
\|\phi_1\|^{\alpha}_{L^{\alpha}(\G)}=C_{1}^{\alpha-1} \frac{\alpha b }{\alpha b- Q(1-\alpha)}\,,
\]
where $C_1$ is given in \eqref{C_1}. 

Let us now give a lower bound for the relative  entropy of $u$ and $\phi_1$. By Jensen's inequality for the same probability measure, and since $\alpha<1$, we have
\begin{equation}
    \label{2}
    \begin{split}
    \int_{\G}u(x)^{\alpha} \log \frac{\phi_1(x)}{u(x)}\,dx &\ = \frac{1}{\alpha-1} \int_{\G} \frac{\|u\|_{L^{\alpha}(\G)}^{\alpha}}{\|u\|_{L^{\alpha}(\G)}^{\alpha}} u(x)^{\alpha} \log \frac{1+|x|^b}{u(x)^{\alpha-1}}\,dx+\|u\|_{L^{\alpha}(\G)}^{\alpha} \log C_1  \\ & 
    \geq \frac{\|u\|_{L^{\alpha}(\G)}^{\alpha}}{\alpha-1} \log \left(\frac{1}{\|u\|_{L^{\alpha}(\G)}^{\alpha}}+\frac{1}{\|u\|_{L^{\alpha}(\G)}^{\alpha}}\int_{\G}|x|^b u(x)\,dx\right) \\ &
    +\|u\|_{L^{\alpha}(\G)}^{\alpha} \log C_1\,,
    \end{split}
\end{equation}
where we have used the fact that $\|u\|_{L^1(\G)}=1$. A combination of \eqref{1} and \eqref{2} gives
\begin{equation}
    \label{12}
    \frac{1}{\alpha-1} \log \left(\frac{1}{\|u\|_{L^{\alpha}(\G)}^{\alpha}}+\frac{1}{\|u\|_{L^{\alpha}(\G)}^{\alpha}}\int_{\G}|x|^b u(x)\,dx \right)+\log C_1 \leq \frac{1}{\alpha} \log \frac{\|\phi_1\|_{L^{\alpha}(\G)}^{\alpha}}{\|u\|_{L^{\alpha}(\G)}^{\alpha}}\,.
\end{equation}
Now, let us apply inequality \eqref{12} to the function $\tilde{u}(x):=\lambda^{Q}u(D_\lambda(x))$. To this end, first we note that 
\begin{equation}
    \label{left12}
       \frac{1}{\|\tilde{u}\|_{L^{\alpha}(\G)}^{\alpha}}+\frac{1}{\|\tilde{u}\|_{L^{\alpha}(\G)}^{\alpha}}\int_{\G}|x|^b \tilde{u}(x)\,dx=\frac{\lambda^{-Q(\alpha-1)}}{\|u\|_{L^{\alpha}(\G)}^{\alpha}}+\frac{\lambda^{-Q(\alpha-1)-b}}{\|u\|_{L^{\alpha}(\G)}^{\alpha}}\int_{\G}|x|^b u(x)\,dx\,,
\end{equation}
since 
\begin{eqnarray*}
 \|\tilde{u}\|_{L^{\alpha}(\G)}^{\alpha} & = & \int_{\G}|\lambda^{Q}u(D_\lambda(x))|^\alpha\,dx \\
  & \stackrel{y=D_\lambda(x)}= & \int_{\G}\lambda^{Q\alpha}|u(y)|^\alpha (\lambda^{-Q}dy) \\
  & = & \lambda^{Q(\alpha-1)}\|u\|^{\alpha}_{L^{\alpha}(\G)}\,.
\end{eqnarray*}
Similarly we have
\begin{equation}
    \label{right12}
    \frac{1}{\alpha} \log \frac{\|\phi_1\|^{\alpha}_{L^{\alpha}(\G)}}{\|\tilde{u}\|^{\alpha}_{L^{\alpha}(\G)}}=\frac{1}{\alpha} \log \frac{\lambda^{-Q(\alpha-1)}\|\phi_1\|^{\alpha}_{L^{\alpha}(\G)}}{\|u\|^{\alpha}_{L^{\alpha}(\G)}}=\frac{1}{\alpha-1} \log \frac{\lambda^{-Q(\alpha-2+\frac{1}{\alpha})}\|\phi_1\|^{\alpha-1}_{L^{\alpha}(\G)}}{\|u\|^{\alpha-1}_{L^{\alpha}(\G)}}\,.
\end{equation}
Hence, with the use of \eqref{left12} and \eqref{right12}, inequality \eqref{12} for the function $\tilde{u}$ reads as follows
\begin{equation*}
    \begin{split}
    &\ \frac{1}{\alpha-1}  \log\left(\frac{\lambda^{-Q(\alpha-1)}}{\|u\|_{L^{\alpha}(\G)}^{\alpha}}+\frac{\lambda^{-Q(\alpha-1)-b}}{\|u\|_{L^{\alpha}(\G)}^{\alpha}}\int_{\G}|x|^b u(x)\,dx \right) +\log C_1  \\&
    \leq  \frac{1}{\alpha-1} \log \frac{\lambda^{-Q(\alpha-2+\frac{1}{\alpha})}\|\phi_1\|^{\alpha-1}_{L^{\alpha}(\G)}}{\|u\|^{\alpha-1}_{L^{\alpha}(\G)}}
    \,.
    \end{split}
\end{equation*}
Using the properties of the logarithm, the latter inequality can be rearranged as follows
\begin{equation}\label{extra}
    \begin{split}
        &\ \frac{1}{\alpha-1} \log \left(\lambda^{Q\left( \frac{1}{\alpha}-1\right)}+\lambda^{Q\left(\frac{1}{\alpha}-1 \right)-b}\int_{\G}|x|^b u(x)\,dx \right)\\&
       \leq \frac{1}{\alpha-1} \log\left(\frac{C_{1}^{1-\alpha} \|\phi_1\|^{\alpha-1}_{L^{\alpha}(\G)}}{\|u\|_{L^{\alpha}(\G)}^{-1}} \right)\,,
    \end{split}
\end{equation}
and since
\[
\frac{1}{\alpha-1} \log\left(\frac{C_{1}^{1-\alpha} \|\phi_1\|^{\alpha-1}_{L^{\alpha}(\G)}}{\|u\|_{L^{\alpha}(\G)}^{-1}} \right)=- \log \left(\frac{C_1}{\|\phi_1\|_{L^{\alpha}(\G)}}\|u\|_{L^{\alpha}(\G)}^{\frac{1}{1-\alpha}} \right)
\]
multiplying \eqref{extra} by $-1$ we get 
\begin{equation}
    \label{anyla}
    \log \left(\frac{C_1}{\|\phi_1\|_{L^{\alpha}(\G)}}\|u\|_{L^{\alpha}(\G)}^{\frac{1}{1-\alpha}} \right) \leq \frac{1}{1-\alpha} \log \left(\lambda^{Q\left( \frac{1}{\alpha}-1\right)}+\lambda^{Q\left(\frac{1}{\alpha}-1 \right)-b}\int_{\G}|x|^b u(x)\,dx \right)\,,
\end{equation}
where the last inequality holds true for any $\lambda>0$. In view of minimising the right-hand side of \eqref{anyla} over $\lambda>0$, we set 
\[
M(\lambda):= \lambda^{Q\left( \frac{1}{\alpha}-1\right)}+\lambda^{Q\left(\frac{1}{\alpha}-1 \right)-b}\int_{\G}|x|^b u(x)\,dx\,.
\]
If $\lambda^*$ minimises $M(\lambda)$, then $ M'(\lambda^*)=0$, where $ M'(\lambda^*)$ is given by 
\[
   Q\left(\frac{1}{\alpha}-1\right)(\lambda^{*})^{Q(\frac{1}{\alpha}-1)-1}+\left[Q\left(\frac{1}{\alpha}-1\right)-b\right](\lambda^*)^{Q(\frac{1}{\alpha}-1)-b-1}\int_{\mathbb{G}}|x|^b u(x)dx\,.
\]
 One can easily check that for
\begin{align*}
    \lambda^*=\left\{\left(\frac{\alpha b-Q(1-\alpha)}{Q(1-\alpha)}\right)\int_{\mathbb{G}}|x|^b u(x)dx\right\}^\frac{1}{b}
\end{align*}
we have $ M'(\lambda^*)=0$, and we we can compute
\begin{equation*}
\begin{split}
M(\lambda^*) &\ = (\lambda^{*})^{Q\left( \frac{1}{\alpha}-1\right)}\left[1+ (\lambda^{*})^{-b}\int_{\G}|x|^b u(x)\,dx\right] \\&
= \left\{\left(\frac{\alpha b-Q(1-\alpha)}{Q(1-\alpha)}\right)\int_{\mathbb{G}}|x|^b u(x)dx\right\}^{\frac{Q}{b}\left(\frac{1}{\alpha}-1 \right)} \\&
\times \left[1+\frac{Q(1-\alpha)}{\alpha b-Q(1-\alpha)}\left(\int_{\G}|x|^bu(x)\,dx \right)^{-1}\int_{\G}|x|^b u(x)\,dx \right] \\&
=\frac{\alpha b}{\alpha b-Q(1-\alpha)}\left\{\left(\frac{\alpha b-Q(1-\alpha)}{Q(1-\alpha)}\right)\int_{\mathbb{G}}|x|^b u(x)dx\right\}^{\frac{Q}{b}(\frac{1}{\alpha}-1)}\,.
\end{split}
\end{equation*}
Thus, inequality \eqref{anyla} for $\lambda=\lambda^*$ becomes
\begin{align}
   \nonumber&\log\left(\frac{C_{1}\|u\|^{\frac{1}{1-\alpha}}_{L^{\alpha}(\G)}}{\|\phi_1\|_{L^{\alpha}(\G)}}\right)\le \frac{1}{(1-\alpha)}\log\bigg[\left(\frac{\alpha b}{\alpha b-Q(1-\alpha)}\right)\times \\&\nonumber\quad\quad\quad\quad\quad\quad\quad\left\{\left(\frac{\alpha b-Q(1-\alpha)}{Q(1-\alpha)}\right)\int_{\mathbb{G}}|x|^b u(x)dx\right\}^{\frac{Q}{b}(\frac{1}{\alpha}-1)}\bigg]\,,
\end{align}
which in turn implies that 
\begin{align}\label{beforelog}
\|u\|_{L^{\alpha}(\G)}^\frac{\alpha}{1-\alpha}\le \frac{\|\phi_1\|^{\alpha}_{L^{\alpha}(\G)}}{C_{1}^{\alpha}}  \left(\frac{\alpha b}{\alpha b-Q(1-\alpha)}\right)^{\frac{\alpha}{1-\alpha}}\left\{\left(\frac{\alpha b-Q(1-\alpha)}{Q(1-\alpha)}\right)\int_{\mathbb{G}}|x|^b u(x)dx\right\}^{\frac{Q}{b}}\,.
\end{align}
By the monotonicity of the function $\log x$, the latter inequality implies \eqref{eq1.1} for $\alpha \in (0,1)$, with  $A_{\alpha,Q,b}$ given by
\begin{eqnarray*}\label{Aabq1}
    A_{\alpha,Q,b} & = &\frac{\|\phi_1\|^{\alpha}_{L^{\alpha}(\G)}}{C_{1}^{\alpha}}   \left(\frac{\alpha b}{\alpha b-Q(1-\alpha)}\right)^{\frac{\alpha}{1-\alpha}}\left(\frac{\alpha b-Q(1-\alpha)}{Q(1-\alpha)}\right)^\frac{Q}{b}\\
    & = & C^{-1}_{1}\left(\frac{\alpha b}{\alpha b-Q(1-\alpha)}\right)^{\frac{1}{1-\alpha}}\left(\frac{\alpha b-Q(1-\alpha)}{Q(1-\alpha)}\right)^\frac{Q}{b}\,,
\end{eqnarray*}
where we have used the explicit formula for $\|\phi_1\|^{\alpha}_{L^{\alpha}(\G)}$ given by \eqref{phia}. 
\medskip

Let us now treat the case where $\alpha>1$. Let us also consider the auxiliary function $\phi_2$ given by 
\[
\phi_2(x)=C_2(1-|x|^b)_{+}^{\frac{1}{1-\alpha}}\,,
\]
with $C_2$  given in \eqref{C_2corr}. Using Jensen's inequality for the concave function $G(x)=-x\log x$ and with the probability measure $\frac{u(x)^\alpha}{\|u\|_{L^{\alpha}(\G)}^{\alpha}}dx$ we estimate
\begin{eqnarray}
    \int_{\G}u(x)\phi_2(x)^{\alpha-1}\log \frac{\phi_2(x)^{\alpha-1}}{u(x)^{\alpha-1}}\,dx & = & \int_{\mathbb{G}} u(x)^{\alpha}\frac{\phi_2(x)^{\alpha-1}}{u(x)^{\alpha-1}}
\log\frac{\phi_2(x)^{\alpha-1}}{u(x)^{\alpha-1}}dx \nonumber \\
& = & -\int_{\mathbb{G}} u(x)^\alpha\,G\left(\frac{\phi_2(x)^{\alpha-1}}{u(x)^{\alpha-1}}\right)\,dx \nonumber\\
& = &  -\|u\|^{\alpha}_{L^{\alpha}(\G)}\int_{\mathbb{G}} \frac{u(x)^\alpha}{\|u\|_{L^{\alpha}(\G)}^\alpha}G\left(\frac{\phi_2(x)^{\alpha-1}}{u(x)^{\alpha-1}}\right)dx\nonumber \\
& \geq & -\|u\|^{\alpha}_{L^{\alpha}(\G)} G\left(\frac{1}{\|u\|_{L^{\alpha}(\G)}^{\alpha}}\int_{\mathbb{G}}u(x)\,\phi_2(x)^{\alpha-1}dx\right)\nonumber\,.
\end{eqnarray}
Now, if we denote by $K=K(u, \phi_2,\alpha)$ the quantity 
\begin{equation}\label{defK}
K:=\int_{\G}u(x)\phi_2(x)^{\alpha-1}\,dx=\int_{\G}u(x)(C_2(1-|x|^b)_{+}^{\frac{1}{1-\alpha}})^{\alpha-1}\,dx < \infty\,,
\end{equation}
then, by the above computations, and the definition of the function $G$, we get 
\begin{equation}
    \label{1a}
      \int_{\G}u(x)\phi_2(x)^{\alpha-1}\log \frac{\phi_2(x)^{\alpha-1}}{u(x)^{\alpha-1}}\,dx \geq K \log \frac{K}{\|u\|^{\alpha}_{L^{\alpha}(\G)}}\,.
\end{equation}
 On the other hand, by the properties of the logarithm we have  
 \begin{equation}
     \label{obs}
      \int_{\G}u(x)\phi_2(x)^{\alpha-1}\log \frac{\phi_2(x)^{\alpha-1}}{u(x)^{\alpha-1}}\,dx= (\alpha-1)K \int_{\G} \frac{u(x)\phi_2(x)^{\alpha-1}}{K} \log \frac{\phi_2(x)}{u(x)}\,dx\,.
 \end{equation}
 Another application of Jensen's inequality to the right-hand side of \eqref{obs} yields 
 \begin{equation}
     \label{2a}
      \int_{\G}K^{-1}u(x)\phi_2(x)^{\alpha-1}\log \frac{\phi_2(x)^{\alpha-1}}{u(x)^{\alpha-1}}\,dx \leq  (\alpha-1)\log (K^{-1}\|\phi_2\|_{L^{\alpha}(\G)}^{\alpha})\,,
 \end{equation}
 where $\|\phi_2\|_{L^{\alpha}(\G)}^{\alpha}$ has been computed in the Appendix, see \eqref{phib}, as 
 \[
 \|\phi_2\|_{L^{\alpha}(\G)}^{\alpha}= \|\phi_2\|_{L^{\alpha}(\G)}^{\alpha}=C_{2}^{\alpha-1}\frac{\alpha b }{b \alpha+Q(\alpha-1)}\,.
 \]
 Hence by \eqref{obs} and \eqref{2a} we get 
 \begin{equation}
     \label{2aa}
          \int_{\G}u(x)\phi_2(x)^{\alpha-1}\log \frac{\phi_2(x)^{\alpha-1}}{u(x)^{\alpha-1}}\,dx \leq (\alpha-1)K \log \frac{\|\phi_2\|_{L^{\alpha}(\G)}^{\alpha}}{K}\,. 
 \end{equation}
 A combination of \eqref{1a} and \eqref{2aa} gives 
 \begin{equation}
     \label{12a}
     \log \frac{K}{\|u\|_{L^{\alpha}(\G)}^{\alpha}} \leq (\alpha-1) \log \frac{\|\phi_2\|_{L^{\alpha}(\G)}^{\alpha}}{K}\,,
 \end{equation}
 or equivalently 
 \[
 \alpha \log \frac{K}{\|\phi_2\|_{L^{\alpha}(\G)}^{\alpha-1}\|u\|_{L^{\alpha}(\G)}}\leq 0\,,
 \]
 and since $\log x \leq 0$ for $x\leq 1$, the latter implies that 
 \begin{equation}
     \label{Q1}
     K \leq \|\phi_2\|_{L^{\alpha}(\G)}^{\alpha-1}\|u\|_{L^{\alpha}(\G)}\,.
 \end{equation}
Now, by \eqref{defK}, and since clearly $(1-|x|^b)_{+}\geq (1-|x|^b)$, we have 
\begin{eqnarray}
    \label{Q2}
    C^{\alpha-1}_{2} \int_{\G} (1-|x|^b)u(x)\,dx \leq K\,.
\end{eqnarray}
A combination of \eqref{Q1} and \eqref{Q2} gives 
\[
C^{\alpha-1}_{2} \int_{\G} (1-|x|^b)u(x)\,dx \leq \|\phi_2\|_{L^{\alpha}(\G)}^{\alpha-1}\|u\|_{L^{\alpha}(\G)}\,,
\]
or, using the fact that $\|u\|_{L^1(\G)}=1$, the latter can be simplified as follows 
\begin{equation}
    \label{Qf}
    C_{2}^{\alpha-1}\leq C_{2}^{\alpha-1}\int_{\G}|x|^bu(x)\,dx+\|\phi_2\|_{L^{\alpha}(\G)}^{\alpha-1}\|u\|_{L^{\alpha}(\G)}\,.
\end{equation}
For $\tilde{u}$ as before, the right-hand side of \eqref{Qf} becomes 
\[
\lambda^{-b}C_{2}^{\alpha-1}\int_{\G}|x|^b u(x)\,dx+\lambda^{Q\left(1-\frac{1}{\alpha} \right)}\|\phi_2\|_{L^{\alpha}(\G)}^{\alpha-1}\|u\|_{L^{\alpha}(\G)}\,.
\]
Therefore, for any $\lambda>0$ the following inequality holds true
\begin{equation}
    \label{anyl1}
   C_{2}^{\alpha-1}\leq  \lambda^{-b}C_{2}^{\alpha-1}\int_{\G}|x|^b u(x)\,dx+\lambda^{Q\left(1-\frac{1}{\alpha} \right)}\|\phi_2\|_{L^{\alpha}(\G)}^{\alpha-1}\|u\|_{L^{\alpha}(\G)}:=N(\lambda)\,.
\end{equation}
Our next aim is, as we did above, to minimise the quantity $N(\lambda)$ over $\lambda$. Now, if we set $I=\int_{\G}|x|^bu(x)\,dx$, then we can write 
\[
N(\lambda)= \lambda^{-b}C_{2}^{\alpha-1}I+\lambda^{Q\left(1-\frac{1}{\alpha} \right)}\|\phi_2\|_{L^{\alpha}(\G)}^{\alpha-1}\|u\|_{L^{\alpha}(\G)}\,,
\]
and we have
\[
N'(\lambda)= (-b)\lambda^{-b-1}C_{2}^{\alpha-1}I+Q\left(1-\frac{1}{\alpha} \right)\lambda^{Q\left(1-\frac{1}{\alpha} \right)-1}\|\phi_2\|_{L^{\alpha}(\G)}^{\alpha-1}\|u\|_{L^{\alpha}(\G)}\,.
\]
It is then not difficult to check that $N'(\lambda^{**})=0$, for 
\[
\lambda^{**}= \left(\frac{b \alpha C_{2}^{\alpha-1}I}{\|\phi_2\|_{L^{\alpha}(\G)}^{\alpha-1}Q(\alpha-1)\|u\|_{L^{\alpha}(\G)}} \right)^{\frac{\alpha}{b \alpha+Q(\alpha-1)}}\,.
\]
A direct substitution then gives
\begin{equation}\label{Nla}
\begin{split}
    N(\lambda^{**}) & \ =  \left(\frac{b \alpha C_{2}^{\alpha-1}I}{\|\phi_2\|_{L^{\alpha}(\G)}^{\alpha-1}Q(\alpha-1)\|u\|_{L^{\alpha}(\G)}} \right)^{\frac{-b\alpha}{b \alpha+Q(\alpha-1)}}C_{2}^{\alpha-1}I  \\ &
     +  \left(\frac{b \alpha C_{2}^{\alpha-1}I}{\|\phi_2\|_{L^{\alpha}(\G)}^{\alpha-1}Q(\alpha-1)\|u\|_{L^{\alpha}(\G)}} \right)^{\frac{\alpha Q}{b \alpha+Q(\alpha-1)}\left(1-\frac{1}{\alpha} \right)}\|\phi_2\|_{L^{\alpha}(\G)}^{\alpha-1}\|u\|_{L^{\alpha}(\G)}\,.
   % & = & S I^{\frac{Q(\alpha-1)}{Q(\alpha-1)+\alpha b}} \|u\|_{L^{\alpha(\G)}}^{\frac{\alpha %b}{Q(\alpha-1)+\alpha b}},
\end{split}
\end{equation}
Observe that since
\[
-\left(\frac{-b\alpha}{b \alpha+Q(\alpha-1)} \right)=-\left[\frac{\alpha Q}{b \alpha+Q(\alpha-1)}\left(1-\frac{1}{\alpha} \right) \right]+1=\frac{\alpha b}{Q(\alpha-1)+\alpha b}\,,
\]
and 
\[
\left[\frac{-b\alpha}{b \alpha+Q(\alpha-1)} \right]+1=\frac{\alpha Q}{b \alpha+Q(\alpha-1)}\left( 1-\frac{1}{\alpha}\right)=\frac{Q(\alpha-1)}{Q(\alpha-1)+\alpha b}
\]
the quantities $I^{\frac{Q(\alpha-1)}{Q(\alpha-1)+\alpha b}}$ and $\|u\|_{L^{\alpha(\G)}}^{\frac{\alpha b}{Q(\alpha-1)+\alpha b}}$ are common multiplies of the summands in \eqref{Nla}, and we can write
\begin{equation}
    \label{Nlaa}
     N(\lambda^{**})=S I^{\frac{Q(\alpha-1)}{Q(\alpha-1)+\alpha b}} \|u\|_{L^{\alpha(\G)}}^{\frac{\alpha b}{Q(\alpha-1)+\alpha b}}
\end{equation}
where we have denoted by $S=S(\alpha,b,Q)$ the following quantity 
\begin{eqnarray}\label{S}
    S & = & \left(\frac{\alpha b}{Q\left(\alpha-1\right)}\right)^{\frac{-\alpha b}{Q(\alpha-1)+\alpha b}} C^{\frac{{Q(\alpha-1)^2}}{Q(\alpha-1)+\alpha b}}_{2}\,\|\phi_2\|_{L^{\alpha}(\G)}^{\frac{\alpha b(\alpha-1)}{Q(\alpha-1)+\alpha b}}\left(\frac{\alpha b+Q(\alpha-1)}{Q(\alpha-1)}\right)\nonumber\\
    & = & \left(\frac{\alpha b+Q(\alpha-1)}{\alpha b}\right)^{\frac{b+Q(\alpha-1)}{\alpha b+Q(\alpha -1)}}\left(\frac{\alpha b}{Q\left(\alpha-1\right)}\right)^\frac{Q(\alpha-1)}{\alpha b+Q(\alpha-1)}C^{\frac{{(Q+b)(\alpha-1)^2}}{Q(\alpha-1)+\alpha b}}_{2}\,,
\end{eqnarray}
and for the last equality we have used the expression \eqref{phia1} for $\|\phi_2\|^{\alpha}_{L^{\alpha}(\G)}$. 

Hence by \eqref{anyl1} and \eqref{Nlaa} we get 
\begin{align*}
  &1\le SC^{1-\alpha}_{2} \left(\int_{\mathbb{G}}|x|^bu(x) dx\right)^{\frac{Q(\alpha-1)}{Q(\alpha-1)+\alpha b}} \|u\|_{L^{\alpha}(\G)}^{\frac{\alpha b}{Q(\alpha-1)+\alpha b}}\,,
\end{align*}
where we have substituted the expression for $I$, and the latter implies that 
\begin{align*}
  &- \log\left(\|u\|_{L^{\alpha}(\G)}^{\frac{\alpha b}{Q(\alpha-1)+\alpha b}} \right)\le \log \left[SC^{1-\alpha}_{2} \left(\int_{\mathbb{G}}|x|^bu(x) dx\right)^{\frac{Q(\alpha-1)}{Q(\alpha-1)+\alpha b}}\right]\,.
\end{align*}
Hence we get  
\begin{align*}
  &-\frac{b}{Q(\alpha-1)+\alpha b}\log \int_{\G}u(x)^{\alpha}\,dx\le \ \frac{Q(\alpha-1)}{Q(\alpha-1)+\alpha b} \log\left([SC^{1-\alpha}_{2}]^{\frac{Q(\alpha-1)+\alpha b}{Q(\alpha-1)}} \int_{\mathbb{G}}|x|^bu(x)\,dx\right)\,,
\end{align*}
or after simplifications 
\begin{align}
    \frac{1}{1-\alpha} \log \int_\G u(x)^\alpha\, dx \leq \frac{Q}{b} \log \Bigg( [C^{1-\alpha}_{2} S]^{\frac{b\alpha}{Q(\alpha-1)}+1} \int_\G |x|^b u(x)  dx \Bigg),
\end{align} and we have obtained \eqref{eq1.1} for $\alpha>1$
with  $A_{\alpha,Q,b}$ given by 
\begin{equation}
\begin{split}
    A_{\alpha,Q,b} & \ = [C^{1-\alpha}_{2} S]^{\frac{b\alpha}{Q(\alpha-1)}+1}  \\&
 = \left[\left(\frac{\alpha b+Q(\alpha-1)}{\alpha b}\right)^{\frac{b+Q(\alpha-1)}{\alpha b+Q(\alpha -1)}}\left(\frac{\alpha b}{Q\left(\alpha-1\right)}\right)^\frac{Q(\alpha-1)}{\alpha b+Q(\alpha-1)} \right]^{\frac{b\alpha}{Q(\alpha-1)}+1} \\&
 \times C_{2}^{\left\{(1-\alpha)+\frac{(Q+b)(\alpha-1)^2}{Q(\alpha-1)+\alpha b}\right\} \left\{ \frac{b\alpha}{Q(\alpha-1)}+1\right\}}  \\ &
     =\frac{\alpha b}{Q(\alpha-1)}\left(\frac{\alpha b +Q(\alpha-1)}{\alpha b} \right)^{\frac{Q(\alpha-1)+b}{Q(\alpha-1)}}\left(\frac{b}{|\mathfrak{S}|}\frac{\Gamma\left(\frac{\alpha}{\alpha-1}+\frac{Q}{b}\right)}{\Gamma(\frac{\alpha}{\alpha-1})\Gamma(\frac{Q}{b})} \right)^{-\frac{b}{Q}}\,,
     \end{split}
\end{equation}
where we have used the expressions for $S$ and $C_2$ given in \eqref{S} and \eqref{C_2corr}, respectively. 
This completes the proof of the Theorem \ref{mainthm}.
 \end{proof}
 The proving procedure that we followed in Theorem \ref{mainthm} show that the inequality \eqref{eq1.1} is sharp. Indeed we have the following result:
 \begin{cor}\label{cor.sharp}
     The constant $A_{\alpha,Q,b}$ in the inequality \eqref{eq1.1} is sharp. In particular, when $\alpha \in (0,1)$, the inequality \eqref{eq1.1} holds true as as equality for (up to $L^1(\G)$-scaling) $u=\phi_1$ where $\phi_1(x)=C_1(1+|x|^b)^{\frac{1}{\alpha-1}}$ and $C_1$ is given by \eqref{C_1}. Similarly, when $\alpha>1$, the inequality \eqref{eq1.1} holds true as an equality for (up to $L^{1}(\G)$-scaling) $u=\phi_2$, where $\phi_2(x)=C_2 (1+|x|^b)^{\frac{1}{\alpha-1}}_{+}$, with $C_2$  given in \eqref{C_2corr}.
 \end{cor}

%\begin{rem}\label{rem.alpa1}
%Corollary \ref{cor.sharp} yields an alternative expression for the constant $A_{\alpha,Q,b}$ for the case where $\alpha \in (0,1)$. Indeed, since inequality \eqref{beforelog} is equivalent to \eqref{eq1.1} as in Theorem \ref{mainthm}, we have 
%\[
%\|\phi_1\|_{L^{\alpha}(\G)}^{\frac{\alpha}{1-\alpha}}=A_{\alpha,Q,b}\left\{ \int_{\G}|x|^b\phi_1(x)\,dx\right\}^{\frac{Q}{b}}\,.
%\]
%Hence by substituting $\|\phi_1\|_{L^{\alpha}(\G)}$ as in \eqref{phia}, and the integral $\int_{\G}|x|^b\phi_1(x)\,dx$ as in \eqref{phia1} we get the following expression for $A_{\alpha,Q,b}$,
%\[
%A_{\alpha,Q,b}=\left(\frac{\alpha b}{Q(1-\alpha)}\right)\left(\frac{\alpha b}{\alpha b-Q(1-\alpha)}\right)^{\frac{b}{Q(1-\alpha)}-1}C^{\frac{-b}{Q}}_{1}\,,
%\]
%where $C_1$ is given in \eqref{C_1}.
%\end{rem}
\begin{rem}\label{rem.notsharp}
    Using simple mathematical arguments, it is easy to check that the left-hand side of the inequality \eqref{eq1.1}, i.e., of the R\'enyi entropy, when $\alpha \rightarrow 1$ approximates the Shannon entropy, i.e., the quantity 
    \[
    -\int_{\G}u(x) \log u(x)\,dx\,.
    \]
    In \cite{CKR} the authors proved the optimal Shannon inequality on homogeneous Lie groups that reads as follows
    \begin{equation}\label{Shann.CG}
     -\int_{\G}u(x) \log u(x)\,dx \leq \frac{Q}{2} \log \left(C_\G \int_\G |x|^2 u(x)\,dx \right)\,,
    \end{equation}
    where the best constant $C_\G$ is given by 
    \begin{equation}\label{CG}
    C_\G=\frac{2e}{Q}\left( \frac{|\mathfrak{S}|\Gamma\left(\frac{Q}{2} \right)}{2}\right)^{\frac{2}{Q}}\,,
    \end{equation}
     and we have $C_{\R}=\frac{2e\pi}{n}$, as expected, since the constant $C_{\G}$ is optimal.
  \end{rem}

 In the next result, by taking the limit of the R\'enyi entropy, we also obtain the so-called Shannon inequality on homogeneous Lie groups. As expected, the sharpness of the constant in the inequality \eqref{eq1.1}  implies an optimal bound for the Shannon entropy in the aforesaid setting which is exactly the quantity on the right-hand side of \eqref{Shann.CG}.
 \begin{thm}\label{thm.Shann}
     Let $\G$ be a homogeneous Lie group of homogeneous dimension $Q$, and suppose that $|\cdot|$ is a homogeneous quasi-norm on $\G$. Then Shannon inequality reads as follows
     \begin{equation}
         \label{Shan}
          -\int_{\G}u(x) \log u(x)\,dx \leq  \frac{Q}{2} \log \left(C_\G \int_\G |x|^2 u(x)\,dx \right)\,,
     \end{equation}
     where 
     \[
   C_{\G}= \frac{2e}{Q}\left( \frac{|\mathfrak{S}|\Gamma\left(\frac{Q}{2} \right)}{2}\right)^{\frac{2}{Q}}\,.
     \]
 \end{thm}
 \begin{proof}
     We first consider the case where $\alpha>1$. By Remark \ref{rem.notsharp} it is enough to show that $\lim_{\alpha \rightarrow 1}A_{\alpha,Q,2}=C_{\G}$, where $A_{\alpha,Q,2}$ is given by Theorem \ref{mainthm} for $\beta=2$, and $\alpha>1$.  If we set $c=\frac{\alpha}{\alpha-1}$, then $c \rightarrow \infty$, when $\alpha \rightarrow 1$, and we have 
    \[
    A_{c(\alpha),Q,2}=\frac{2c}{Q}\left(1+\frac{Q}{2c} \right)^{\frac{2c}{\alpha Q}+1} \left(\frac{2}{|\mathfrak{S}|}\frac{\Gamma\left( c+\frac{Q}{2}\right) }{\Gamma\left(c \right)\Gamma\left(\frac{Q}{2} \right)}\, \right)^{-\frac{2}{Q}}\,.
    \]
    To calculate the quantity $\left(\frac{\Gamma\left(c \right)}{\Gamma\left(c+\frac{Q}{2} \right)}\, \right)^{\frac{2}{Q}}$, when $c \gg 1$, we will use the Stirling approximation formula:
    \[
    \Gamma(x)\simeq \sqrt{2 \pi} e^{-x}x^{x-1/2}\,,\quad \text{when} \quad x \gg 1\,.
    \]
    We have 
    \begin{eqnarray*}
     \left(\frac{\Gamma\left(c \right)}{\Gamma\left(c+\frac{Q}{2} \right)}\, \right)^{\frac{2}{Q}} & \simeq & \left[ \frac{\sqrt{2 \pi}e^{-c}c^{c-1/2}}{\sqrt{2 \pi} e^{-c-Q/2}(c+Q/2)^{c+Q/2-1/2}}\right]^{\frac{2}{Q}}\\
    & = & e c^{\frac{2c}{Q}-\frac{1}{Q}} (c+Q/2)^{-1-\frac{2c}{Q}+\frac{1}{Q}}\,.
    \end{eqnarray*}
    Since
    \[
    \lim_{c \rightarrow \infty, \alpha \rightarrow 1}\left(1+\frac{Q}{2c} \right)^{\frac{2c}{\alpha Q}+1} =1\,,
    \]
    we have 
    \begin{eqnarray*}
        \lim_{c \rightarrow \infty,\alpha \rightarrow 1}  A_{c(\alpha),Q,2} & = & \frac{2e}{Q}\left( \frac{2}{|\mathfrak{S}|}\Gamma\left(\frac{Q}{2} \right)^{-1}\right)^{-\frac{2}{Q}}  \lim_{c \rightarrow \infty} c c^{\frac{2c}{Q}-\frac{1}{Q}} (c+Q/2)^{-1-\frac{2c}{Q}+\frac{1}{Q}}\\
        & = & \frac{2e}{Q}\left( \frac{2}{|\mathfrak{S}|}\Gamma\left(\frac{Q}{2} \right)^{-1}\right)^{-\frac{2}{Q}}\\
        & = & \frac{2e}{Q}\left( \frac{|\mathfrak{S}|}{2}\Gamma\left(\frac{Q}{2} \right)\right)^{\frac{2}{Q}} \,,
    \end{eqnarray*}
    completing the proof.
 \end{proof}

    \begin{comment}
    Therefore one might expect that the optimal Shannon inequality in the setting of homogeneous Lie groups can be proved by taking the limit of \eqref{eq1.1} as $\alpha \rightarrow 1$. However, this is not the case, since we have 
    \begin{equation}
    \lim_{\alpha \rightarrow 1}A_{\alpha,Q,2}=
    \begin{cases}
    0\,, \quad \rm if \quad \alpha>1\\
    0\,, \quad \rm if \quad \alpha<1\,.
    \end{cases}
    \end{equation}
    Indeed, 
   
    Regarding the Shannon inequality in the general setting of homogeneous Lie groups, this is proved in  \cite{CKR}  using completely different methods from ours, and reads as follows
    \begin{equation}
        \label{Shan}
        -\int_{\G}u(x) \log u(x)\,dx \leq \frac{Q}{2} \log \left(C_{\G} \int_{\G}|x|^2 u(x)\,dx \right)\,,
    \end{equation}
 where the   appearing constant $C_{\G}$ is given by
     \[
     C_{\G}= 4\left( \frac{|\mathfrak{S}|\Gamma\left(\frac{Q}{2} \right)^{2}}{2\Gamma(Q)}\right)^{\frac{2}{Q}}\,.
     \]
     Recall, see e.g. \cite{Suguro}, that the Shannon inequality with best constant in the Euclidean case reads as follows 
     \[
      -\int_{\mathbb{R}^n}u(x) \log u(x)\,dx \leq  \frac{n}{2} \log \left(\frac{2\pi e}{n} \int_{\mathbb{R}^n}|x|^2 u(x)\,dx \right)\,,
     \]
     which in particular means that the constant obtained in \cite{CKR} is not sharp since clearly $C_{\mathbb{R}^n} \neq \frac{2 \pi e }{n}$.  \end{comment}

 The next result applies to a subclass of graded Lie groups when $V_1$ as in \eqref{strat} generates the whole of the Lie algebra. Precisely, if $\mathbb{G}$ is a graded Lie group with Lie algebra $\mathfrak{g}$ that admits a gradation of the form \eqref{strat} and where the elements $\{X_1,\cdots,X_k\}$, that are such that $V_1=\text{span}\{X_1,\cdots,X_k\}$, generate after iterated commutators the whole of $\mathfrak{g}$. 
 
 In the case of stratified Lie groups the elliptic Laplace operator $\Delta$ on $\mathbb{R}^n$ is replaced by the (positive) hypoelliptic operator $\Delta_{\text{sub}}$, that is commonly called a sub-Laplacian on $\mathbb{G}$, defined via 
 \[
 \Delta_{\text{sub}}=-\sum_{i=1}^{k}X_{i}^{2}\,,
 \]
 and we have $\Delta_{\text{sub}}=-\nabla_{H}^{*}\nabla_{H}$, where $\nabla_{H}$ is called the horizontal gradient on $\mathbb{G}$; namely $\nabla_{H}$ is the vector valued operator acting on $f$ via 
 \[
 \nabla_{H}f=(X_1 f,\cdots, X_k f)\,.
 \]
 Let us recall a version of the logarithmic Sobolev inequality in the setting of stratified groups, as appeared in \cite[Corollary 7.1]{CKR1}
 \begin{prop}
     Let $\G$ be a stratified Lie group with homogeneous dimension $Q$. Then, the following log-Sobolev inequality is satisfied
\begin{equation}
    \label{cor.log.sob}
    \int_{\mathbb{G}}|f|^2 \log|f|\,dx \leq \frac{Q}{4}\log \left(A \int_{\mathbb{G}}|\nabla_{H}f|^2\,dx \right)\,,
\end{equation}
for every $f$ such that $\|f\|_{L^2(\G)}=1$, for some constant $A$ with a formula that depends on the group $\G$. 
 \end{prop}
 Let us point out that the constant $A$ is related to the Gagliardo-Nirenberg inequalities in our setting, and cannot be yet explicitly computed in the general setting of a stratified Lie group. However, in the particular case of the Heisenberg group we have $A=\frac{(n!)^{\frac{1}{n+1}}}{\pi n^2}$, see \cite{JL88}, \cite{CKR23}, and when $\G=\mathbb{R}^n$ we have $A=(\pi n^2-2\pi n)^{-1/2}\frac{\Gamma(n)}{\Gamma(n/2)}$, see \cite{Tal}, and in these cases $A$ is optimal. To avoid technicalities we do not include  the explicit formula for $A$ here, and we refer the interested reader to \cite{CKR1,GKR, RTY}. 
 \begin{thm}\label{thm.rev.LS}
      Let $\G$ be a stratified Lie group with homogeneous dimension $Q$, and let $u \in L^1(\G) \cap L^2(\G)$ be a nonnegative function such that $\|u\|_{L^1(\G)}=1.$ Then it holds that 
      \begin{eqnarray}\label{log.sov.11}
      \int_{\G}u(x) \log u(x)\,dx \leq  \frac{Q}{2} \log \left(\frac{A}{4}\int_{\G}\frac{1}{u(x)}|\nabla_{H}u|^2\,dx \right)\,,   
      \end{eqnarray}
      where the constant $A$ is the one appearing in \eqref{cor.log.sob}.
 \end{thm}
 \begin{proof}
    Let $u$ be a function as in the hypothesis. The proof follows by inequality \eqref{cor.log.sob} for $u$  such that $f=u^{\frac{1}{2}}$.
 \end{proof}
 As it happens in the Euclidean setting, see \cite{KOS}, a combination of the logarithmic Sobolev inequality of the form \eqref{log.sov.11} with the Shannon's inequality \eqref{Shan} yields the Heisenberg uncertainty principle in the setting of stratified groups.
 \begin{cor}\label{cor.unc}
     A version of the  Heisenberg uncertainty principle in the setting of stratified groups reads as follows:
       \begin{equation}
          \label{UP1}
           \left( \int_{\G}|x|^2 u(x)\,dx \right) \left(\int_{\G}\frac{1}{u(x)}|\nabla_{H}u|^2\,dx \right)\geq 4C_{\G}^{-1}A^{-1}\,,
      \end{equation}
     for all functions $u$ as in the hypothesis of Theorem \ref{thm.rev.LS}, where  $C_\G$ and $A$ are given  in \eqref{CG}, \eqref{cor.log.sob}, respectively. 
 \end{cor}
 \begin{proof}
     A combination of \eqref{log.sov.11} and \eqref{Shann.CG}  gives 
     \[
     -\frac{Q}{2} \log \left(\frac{A}{4} \int_{\G}\frac{1}{u(x)}|\nabla_{H}u|^2\,dx \right) \leq \frac{Q}{2} \log \left(C_{\G} \int_{\G}|x|^2 u(x)\,dx \right)\,.
     \]
     Therefore, we have
     \[
    -\log (C_{\G}\frac{A}{4}) \leq  \log \left[ \left(\int_{\G}\frac{1}{u(x)}|\nabla_{H}u|^2\,dx \right)\left( \int_{\G}|x|^2 u(x)\,dx\right) \right]\,,
     \]
     and  by the properties of the logarithm, we obtain 
      \[
    \left( \int_{\G}|x|^2 u(x)\,dx \right) \left(\int_{\G}\frac{1}{u(x)}|\nabla_{H}u|^2\,dx \right)\geq 4A^{-1}C_{\G}^{-1}\,,
     \]
    and the proof is complete.     
 \end{proof}

  \section{Appendix} For the proof of Theorem \ref{mainthm} we used the auxiliary functions $\phi_1,\phi_2\in L^1(\mathbb{G})\cap L^1_b(\mathbb{G})$, given by 
  $$\phi_1(x)=C_{1}\left(1+|x|^b\right)^{\frac{1}{\alpha-1}},\quad\text{where}\quad b>Q\left(\frac{1}{\alpha}-1\right), \,\alpha \in (0,1)\,,$$
  \begin{equation}\label{dfn.phi2}
  \phi_2(x)=C_{2}\left(1-|x|^b\right)_+^{\frac{1}{\alpha-1}},\quad\text{where}\quad b>0,\,\alpha>1\,,
  \end{equation}
  where the constants $C_1=C_1(\alpha,b,Q)$ and $C_2=C_2(\alpha,b,Q)$ that we are computing below are such that $\|\phi_i\|_{L^1(\G)}=1$, for $i=1,2$. The computations that follow, is a homogeneous group adaptation of the analogous computation in the Euclidean setting developed in \cite{Suguro}.
  \smallskip

  Using the polar decomposition for the function $\phi_1$ we can write 
 \begin{equation*}
 \begin{split}
\int_{\mathbb{G}}\left(1+|x|^b\right)^{\frac{1}{\alpha-1}}dx &= \int_0^\infty\int_\mathfrak{S}\left(1+r^b\right)^{\frac{1}{\alpha-1}}r^{Q-1}dr\,d\sigma(y)\\&
=|\mathfrak{S}|\int_0^\infty \left(1+r^b\right)^{\frac{1}{\alpha-1}}r^{Q-1}dr\,.
\end{split}
\end{equation*}
If we set $1+r^b=t$, so that  $r=(t-1)^\frac{1}{b}$, then we can further compute 
\begin{equation}\label{beforebeta}
\begin{split}
   \int_{\mathbb{G}}\left(1+|x|^b\right)^{\frac{1}{\alpha-1}}dx &=  |\mathfrak{S}|\int_1^\infty t^{\frac{1}{\alpha-1}}\frac{(t-1)^\frac{(Q-1)}{b}}{b(t-1)^{1-\frac{1}{b}}}dt \\&
  = \frac{|\mathfrak{S}|}{b}\int_1^\infty t^{\frac{1}{\alpha-1}}(t-1)^{\frac{Q}{b}-1}dt\,.
  \end{split}
\end{equation}
To further compute \eqref{beforebeta}, let us first recall the definition of the Beta function \begin{equation}\label{dfn.beta}
    B(x,y)=\int_{0}^{1}t^{x-1}(1-t)^{y-1}dt\,,
\end{equation}
  where $x,y \in \mathbb{C}$, with $\text{Re}(x), \text{Re}(y)>0$. Alternatively, one can obtain the following representation of $B$
  \begin{equation}\label{beta.f}
  B(x,y)=\int_1^\infty w^{-(x+y)}(w-1)^{y-1}dw\,.
  \end{equation}
   For $-(x+y)=\frac{1}{\alpha-1}$ and $y=\frac{Q}{b}$, we have
     $x=\frac{1}{1-\alpha}-\frac{Q}{b}$, so that with the use of \eqref{beta.f} and the property (see e.g. \cite{Suguro}) which relates the Beta function with the Gamma function
     \[
     B(x,y)=\frac{\Gamma(x)\Gamma(y)}{\Gamma(x+y)}\,,\quad x,y \notin \mathbb{Z}_{0}^{-}\,,
     \]
     due to \eqref{beforebeta} and \eqref{beta.f} we have 
     \begin{align} \label{value of C}
      \nonumber&\int_{\mathbb{G}}\left(1+|x|^b\right)^{\frac{1}{\alpha-1}}dx =\frac{|\mathfrak{S}|}{b}\,B\left(\frac{1}{1-\alpha}-\frac{Q}{b},\frac{Q}{b}\right)\\&\quad\quad\quad\quad\quad\quad\quad=\frac{|\mathfrak{S}|}{b}\frac{\Gamma\left(\frac{1}{1-\alpha}-\frac{Q}{b}\right)\Gamma(\frac{Q}{b})}{\Gamma\left(\frac{1}{1-\alpha}\right)}.  
     \end{align} Now, the requirement  $\|\phi_1\|_{L^1(\G)}=1$, yields  
\begin{align}\label{C_1}
C_{1}=  \frac{b}{|\mathfrak{S}|}\frac{\Gamma\left(\frac{1}{1-\alpha}\right)}{\Gamma\left(\frac{1}{1-\alpha}-\frac{Q}{b}\right)\Gamma(\frac{Q}{b})}\,.
\end{align}

  Next we compute the $L^{\alpha}(\G)$-norm of $\phi_1$.  To this end, let us first compute the following integral 
  \begin{eqnarray*}
      \int_{\mathbb{G}}|x|^b \phi_1(x)\,dx & = & C_{1} |\mathfrak{S}|\int_0^\infty r^b(1+r^b)^{\frac{1}{\alpha-1}} r^{Q-1}\,dr \nonumber \\
      & = & C_{1} |\mathfrak{S}|\int_0^\infty (1+r^b)^\frac{1}{\alpha-1} r^{b+Q-1}\,dr \\
      & = & C_{1} \frac{|\mathfrak{S}}{b}|\int_1^\infty  t^{\frac{1}{\alpha-1}} (t-1)^{\frac{Q+b}{b}-1}\,dt\,,
  \end{eqnarray*}
where we have used the change of variables $(1+r^b)=t$ implying that $r=(t-1)^\frac{1}{b}$.
    Now, using \eqref{beta.f} the latter can be expressed in the following form 
    \begin{eqnarray}\label{phia1}
         \int_{\mathbb{G}}|x|^b \phi_1(x)\,dx & = & C_{1} \frac{|\mathfrak{S}|}{b}\,B\left(\frac{1}{1-\alpha}-\frac{Q}{b}-1,\frac{Q}{b}+1\right) \nonumber \\
         & = & C_{1} \frac{|\mathfrak{S}|}{b}\,\frac{\Gamma(\frac{1}{1-\alpha}-\frac{Q}{b}-1)\Gamma\left(\frac{Q}{b}+1\right)}{\Gamma(\frac{1}{1-\alpha})}\,.
    \end{eqnarray}
       Recall the following property of the Gamma function
     \begin{align*}
         z\Gamma(z)=\Gamma(z+1)\,, \quad \text{for all} \quad z\in \mathbb{C}, z \notin -\mathbb{N}\,.
     \end{align*} 
    We then have
     \begin{align}\label{eq11}
        \Gamma\left(\frac{1}{1-\alpha}-\frac{Q}{b}-1\right) =  \frac{\Gamma\left(\frac{1}{1-\alpha}-\frac{Q}{b}\right)}{\left(\frac{1}{1-\alpha}-\frac{Q}{b}-1\right)},   
     \end{align}
     and 
     \begin{align}\label{eq12}
         \Gamma \left(\frac{Q}{b}+1\right)=\frac{Q}{b}\Gamma\left(\frac{Q}{b}\right)\,.
     \end{align}
     Using \eqref{eq11} and \eqref{eq12} in \eqref{phia1}, we obtain
     \begin{eqnarray}\label{B2}
         \int_{\G} |x|^b \phi_1(x)\,dx & = & C_{1} \frac{|\mathfrak{S}|}{b}\,\frac{\frac{\Gamma\left(\frac{1}{1-\alpha}-\frac{Q}{b}\right)}{\left(\frac{1}{1-\alpha}-\frac{Q}{b}-1\right)}\left(\frac{Q}{b}\right)\Gamma\left(\frac{Q}{b}\right)}{\Gamma(\frac{1}{1-\alpha})}\nonumber\\
         & = & \frac{Q(1-\alpha)}{\alpha b-Q(1-\alpha)}\,,
     \end{eqnarray}
     where for the last equality we have used the formula for $C_1$ given in \eqref{C_1}.
     
    Finally using \eqref{B2} and the fact that $\|\phi_1\|_{L^1(\G)}=1$, we get 
    \begin{eqnarray}\label{phia}
        \|\phi_1\|^{\alpha}_{L^{\alpha}(\G)} & = & \int_{\G}C_{1}^{\alpha}(1+|x|^b)^{\frac{\alpha}{\alpha-1}}\,dx \nonumber\\
        & = & \int_{\G}C_{1}^{\alpha-1}C_{1}(1+|x|^b)(1+|x|^b)^{\frac{1}{\alpha-1}}\,dx \nonumber \\
        & = & C_{1}^{\alpha-1} \int_{\G} (1+|x|^b)\phi_{1}(x)\,dx \nonumber\\
        & = & C_{1}^{\alpha-1} \left\{\int_{\G}\phi_1(x)\,dx + \int_{\G}|x|^b \phi_1(x) \right\}\nonumber\\
        & = & C_{1}^{\alpha-1} \left\{1+ \frac{Q(1-\alpha)}{\alpha b -Q(1-\alpha)} \right\}\nonumber\\
        & = & C_{1}^{\alpha-1} \frac{\alpha b }{\alpha b- Q(1-\alpha)}\,.
    \end{eqnarray}
Let us now perform a similar analysis for the function $\phi_2$ as in \eqref{dfn.phi2} where $b>0$ and $ \alpha>1$. As before, we use polar decomposition and we have 
\begin{eqnarray*}
    \int_{\mathbb{G}}(1-|x|^b)_{+}^{\frac{1}{\alpha-1}}\,dx & = & \int_{\mathfrak{S}}\int_{0}^{1}\left(1-r^b\right)^{\frac{1}{\alpha-1}} r^{Q-1}dr\,d\sigma(y) \\
    & = & |\mathfrak{S}| \int_{0}^{1}\left(1-r^b\right)^{\frac{1}{\alpha-1}} r^{Q-1}dr \,.
\end{eqnarray*}

 For $(1-r^b)=t$ we get 
 \begin{eqnarray}\label{alpha>1}
       \int_{\mathbb{G}}(1-|x|^b)_{+}^{\frac{1}{\alpha-1}}\,dx & = & |\mathfrak{S}|\int_0^1 t^{\frac{1}{\alpha-1}}\left(1-t\right)^{\frac{Q-1}{b}}\frac{1}{b}\left(1-t\right)^{\frac{1}{b}-1} dt \nonumber \\
       & = & \frac{|\mathfrak{S}|}{b}\int_0^1 t^{\frac{1}{\alpha-1}}\left(1-t\right)^{\frac{Q}{b}-1}\,dt\,.
 \end{eqnarray}
Arguing as before and using the definition of the  Beta function as in \eqref{dfn.beta} for $x=\frac{\alpha}{\alpha-1}$ and $y=\frac{Q}{b}$, the condition  $\|\phi_2\|_1=1$ yields 
\[
\int_{\mathbb{G}}\phi_2(x)\,dx= \frac{C_{2}}{b}\,|\mathfrak{S}|\frac{\Gamma(\frac{\alpha}{\alpha-1})\Gamma(\frac{Q}{b})}{\Gamma\left(\frac{\alpha}{\alpha-1}+\frac{Q}{b}\right)}=1\,,
\]
and consequently, we get 
\begin{equation}\label{C_2corr}
    C_2= \frac{b}{|\mathfrak{S}|}\frac{\Gamma\left(\frac{\alpha}{\alpha-1}+\frac{Q}{b}\right)}{\Gamma(\frac{\alpha}{\alpha-1})\Gamma(\frac{Q}{b})}\,.
\end{equation}
Following the same arguments as before one can compute the $L^{\alpha}(\G)$-norm of $\phi_2$ and get 
\begin{equation}
    \label{phib}
    \|\phi_2\|_{L^{\alpha}(\G)}^{\alpha}=C_{2}^{\alpha-1}\frac{\alpha b }{b \alpha+Q(\alpha-1)}\,.
\end{equation}

\section{Acknowledgement}
The authors are supported by the FWO Odysseus 1 grant G.0H94.18N: Analysis and Partial Differential Equations and by the Methusalem programme of the Ghent University Special Research Fund (BOF) (Grant number 01M01021).  Michael Ruzhansky is also supported by EPSRC grants EP/R003025/2 and EP/V005529. M. Chatzakou is a postdoctoral fellow of the Research Foundation – Flanders (FWO) under the postdoctoral grant No 12B1223N.

\end{document}